\documentclass[12pt,a4paper]{article}
\usepackage{graphicx}
\usepackage{amssymb,amsmath,amsthm}
\usepackage[mathscr]{eucal}
\usepackage{hyperref}
\usepackage{mathtools}
\usepackage[margin=1in]{geometry}

\newtheorem{thm}{Theorem}[section]

\newtheorem{cor}[thm]{Corollary}
\newtheorem{lem}[thm]{Lemma}
\newtheorem*{notat}{Notation}

\theoremstyle{definition}

\newcommand{\ran}{\text{\normalfont{ran}}}

\newcommand{\fin}{\text{\normalfont{fin}}}
\newcommand{\Part}{\text{\normalfont{Part}}}
\newcommand{\seq}{\text{\normalfont{seq}}}

\newcommand{\fix}{\text{\normalfont{fix}}}
\newcommand{\sym}{\text{\normalfont{sym}}}

\begin{document}

\title{The finite sequences and the partitions whose members are finite of a set}

\author{Palagorn Phansamdaeng and Pimpen Vejjajiva \bigskip\\\small Department of Mathematics and Computer
Science,\\\small Faculty of Science, Chulalongkorn University}

\date{}

\maketitle

\begin{abstract}
    In this paper, we investigate relationships between $|\seq(A)|$ and $|\Part_{\fin}(A)|$ in the absence of the Axiom of Choice, where  $\seq(A)$ is the set of finite sequences of elements in a set $A$ and  $\Part_{\fin}(A)$ is the set of partitions of $A$ whose members are finite. We show that $|\seq(A)|<|\Part_{\fin}(A)|$ if $A$ is Dedekind-infinite and the condition cannot be removed. Moreover, this relationship holds for an arbitrary infinite set $A$ if we restrict $\seq(A)$ to the set of finite sequences with a bounded length.
\end{abstract}


\section{Introduction}\label{sec1}

 Halbeisen and Shelah showed
in \cite[Theorem 3]{HS1994} that \lq\lq$|\fin(A)|<2^{|A|}$ for any infinite set $A$\rq\rq\, is provable in the Zermelo-Fraenkel set theory (ZF), where $\fin(A)$ is the set of finite subsets of $A$ and $2^{|A|}$ is the cardinality of the power set of $A$. With the Axiom of Choice (AC), $|A|=|\fin(A)|=|\seq(A)|$ for any infinite set $A$, where $\seq(A)$ is the set of finite sequences of elements in $A$ (cf. \cite[Theorem 5.19]{H}). However, in the absence of AC, these equalities cannot be proved (cf. \cite{HS2001}). Unlike the relationship between $|\fin(A)|$ and $2^{|A|}$, \lq\lq$|\seq(A)|\ne 2^{|A|}$ for any set $A$ with $|A|>1$\rq\rq\, is the best possible result in ZF (cf. \cite{HS1994} and \cite{HS2001}).

Dawson and Howard showed in \cite{Daw} that, with AC, $|A|!=2^{|A|}$ for any infinite set $A$, where $|A|!$ is the cardinality of the set of permutations on $A$, but any relationship between $|A|!$ and $2^{|A|}$ for an arbitrary infinite set $A$ cannot be proved in ZF. From \cite{SV2019}, Sonpanow and the second author gave analogous results to those of $|\seq(A)|$ and $2^{|A|}$ by showing that \lq\lq$|\seq(A)|\ne |A|!$ for any nonempty set $A$\rq\rq\, is the best possible result in ZF.

We write $\Part(A)$ for the set of partitions of a set $A$ and $\Part_{\fin}(A)$ for the set of partitions of $A$ whose members are finite. If AC is assumed, then $|\Part_{\fin}(A)|=2^{|A|}=|\Part(A)|$ for all infinite sets $A$ (cf.  \cite[Corollary 3.6]{PV} and \cite[XVII.4 Ex.3]{S}). Without AC, it is not hard to show that $2^{|A|}\leq |\Part(A)|$ for any set $A$ with $|A|\geq 5$. Thus, by the above result, $|\fin(A)|<|\Part(A)|$ for any infinite set $A$.
In \cite{PV}, we gave a stronger result by showing in ZF  that $|\fin(A)|<|\Part_{\fin}(A)|$ for any set $A$ with $|A|\geq 5$ while any relationship between $|\Part_{\fin}(A)|$ and $2^{|A|}$ cannot be concluded for an arbitrary infinite set $A$.

 In this paper, we investigate relationships between $|\seq(A)|$ and  $|\Part_{\fin}(A)|$ for infinite sets $A$. We show in ZF that $|\seq(A)|<|\Part_{\fin}(A)|$ if $A$ is Dedekind-infinite and the condition cannot be removed. We also show that this relationship can be obtained for an arbitrary infinite set $A$ if $\seq(A)$ in the statement is restricted to the set of finite sequences with bounded lengths. That is we can show in ZF that $|\seq_{\leq n}(A)|<|\Part_{\fin}(A)|$ for any infinite set $A$ and any natural number $n$, where $\seq_{\leq n}(A)$ is the set of sequences of elements in $A$ with length less than or equal to $n$.

\section{RESULTS IN ZF}\label{sec2}

In this section, we shall work in ZF. For any sets $A$ and $B$, we say that $|A|=|B|$ if there is an explicit bijection between $A$ and $B$,  $|A|\leq|B|$ if there is an explicit injection from $A$ into $B$, and  $|A|<|B|$ if $|A|\leq|B|$ but $|A|\neq|B|$.
A set $A$ is \emph{Dedekind infinite} if $\aleph_0\leq |A|$, otherwise $A$ is \emph{Dedekind finite}.

Apart from the notation introduced in the previous section, we list some that will be used in our work below.

\begin{notat}
\begin{enumerate}
\item  For any set $A$, $\seq^{1-1}(A)$ denotes the set of injective finite sequences of elements in $A$.
\item For a partition $\Pi$ of a set, let $[a]_{\Pi}$ denote the element of $\Pi$ containing $a$.

\item For a sequence $a=\langle a_0,a_1,\dots,a_{n}\rangle$, let $\underline{a}$ denote $\{a_0,a_1,\dots,a_{n}\}$.
\end{enumerate}
\end{notat}

\bigskip

We start with the relationship between $|\seq^{1-1}(A)|$ and $|\Part(A)|$ when $A$ is a finite set.

\begin{thm}
    \label{An>Bn}
For any natural number $n\neq 0$, $|\seq^{1-1}(n)|>|\Part(n)|$.
\begin{proof}
    We proceed by induction on $n$. Suppose the result holds for all $k\leq n$. By recurrences of arrangement number and Bell number,
\begin{align*}
|\seq^{1-1}(n+1)|&=(n+1)|\seq^{1-1}(n)|+1\\&>(n+1)|\Part(n)|\\
&\geq\sum_{k\leq n}\binom{n}{k}|\Part(k)|=|\Part(n+1)|.\qedhere
\end{align*}
\end{proof}
\end{thm}

\begin{cor}\label{A>B}
For any nonempty finite set $A$, $|\seq^{1-1}(A)|>|\Part(A)|$.
\end{cor}

For infinite sets, things turn out to be different. The facts below together with Cantor's theorem show that  $|\seq(A)|<|\Part_{\fin}(A)|$ for all infinite well-ordered sets $A$.

\begin{thm}\cite[Theorem 5.19]{H}
\label{seq}
    For any infinite ordinal $\alpha$, $|\alpha|=|\seq(\alpha)|$.
\end{thm}

\begin{thm}\cite[Theorem 3.5]{PV}
\label{power}
    For any infinite ordinal $\alpha$, $|\Part_{\fin}(\alpha)|=2^{|\alpha|}$.
\end{thm}

 Since AC is equivalent to \lq\lq every set can be well-ordered\rq\rq, we obtain that $|\seq(A)|<|\Part_{\fin}(A)|$ for all infinite sets $A$ if AC is assumed.

What happens in the absence of AC?

We shall show that the relationship holds for Dedekind infinite sets and shall show later in the next section that this assumption cannot be removed. First, we need the following lemma.

\begin{lem}
\label{COR}
For any sets $A$ and $B\subseteq A$, if $|A|=2^{|\alpha|}$ and $|B|\leq|\alpha|$, where $\alpha$ is an infinite ordinal, then $|A|=|A\setminus B|$.
\end{lem}
\begin{proof}
Let $A$ be a set such that $|A|=2^{|\alpha|}$, where $\alpha$ is an infinite ordinal. Since $|\alpha|<2^{|\alpha|}=|A|$, there is $C\subseteq A$ such that $|C|=|\alpha|$. First, we show that $|A|=|A\setminus C|$.

Let $f_A\colon A\rightarrow\mathcal{P}(\alpha)$ and $f_{C}\colon C\rightarrow\alpha$ be bijections and fix a canonical bijection $F\colon \alpha\rightarrow\alpha+\alpha$.

    Define $g\colon \alpha+\alpha\rightarrow\mathcal{P}(\alpha)$ recursively as follows:
    \begin{center}
    $g(\beta)=\begin{cases}f_A\circ f^{-1}_{C}(\beta)&\text{if } \beta<\alpha,\\
    \{\xi<\alpha:F(\xi)<\beta\rightarrow\xi\not\in (g\upharpoonright\beta\circ F)(\xi)\}&\text{otherwise.}
    \end{cases}$
    \end{center}

Clearly, $g\upharpoonright\alpha$ is injective and for each $\alpha\leq\beta<\alpha+\alpha$, if $g(\beta)=(g\upharpoonright\beta\circ F)(\xi)$ for some $\xi<\alpha$, then $F(\xi)<\beta$ and so $\xi\in g(\beta)=(g\upharpoonright\beta\circ F)(\xi)\leftrightarrow \xi\not\in g(\beta)$, a contradiction. Hence $g(\beta)\not\in \ran(g\upharpoonright\beta\circ F)=\ran(g\upharpoonright\beta)$ for all $\alpha\leq\beta<\alpha+\alpha$. Thus $g$ is injective.

Define $G\colon \alpha\rightarrow\mathcal{P}(\alpha)$  by $G(\beta)=g(\alpha+\beta)$ for each $\beta<\alpha$. Clearly, $G$ is injective.
Since $\ran(g\restriction_\alpha)=f_A\circ f_{C}^{-1}[\alpha]=f_A[C]$, $\ran(G)\subseteq\mathcal{P}(\alpha)\setminus f_A[C]$. Hence $f_A^{-1}[\ran(G)]\subseteq A\setminus C$, so
 $f_A^{-1}\circ G \circ f_{C}\colon C\rightarrow A\setminus C$ is injective. Since $|C|$ is an aleph and $|C|\leq|A\setminus C|$, $|A|=|A\setminus C|+|C|=|A\setminus C|$.

Now, let $B\subseteq A$ such that $|B|\leq|\alpha|$. Let $B'= B\cup f_A^{-1}[\alpha]$. Since $|B|\leq|\alpha|=|f_A^{-1}[\alpha]|$, $|B'|=|\alpha|$. Hence, by the above result, $|A|=|A\setminus B'|\leq |A\setminus B|\leq|A|$, and so $|A|=|A\setminus B|$ by the Cantor-Bernstein theorem.
\end{proof}

\begin{thm}
\label{<}
For any Dedekind infinite set $A$, $|\seq (A)|<|\Part_{\fin}(A)|$.
\begin{proof}
Let $A$ be a Dedekind infinite set. Then $|A|= |A\cup\omega|$. We may assume that $A\cap\omega=\emptyset$.
First, we show that $|\seq(A)|\leq|\Part_{\fin} (A\cup\omega)|$.

Define $f\colon \seq (A)\rightarrow\Part_{\fin}(A\cup\omega)$ by stipulating
\[a=\langle a_0,a_1,\dots,a_{n-1}\rangle\mapsto\{\{a_i\}\cup\{j: a_i=a_j\}:i< n\}\cup[A\setminus\underline{a}]^1\cup[\omega\setminus n]^1.\]

To show that $f$ is injective, let $a=\langle a_0,a_1,\dots,a_{m-1}\rangle$, $b=\langle b_0,b_1,\dots,b_{n-1}\rangle\in \seq (A)$ such that $a\neq b$.
If $m=n$, then there is $i<n$ such that $a_i\ne b_i$, so $a_i\in[i]_{f(a)}\setminus[i]_{f(b)}$.
 Otherwise, if $m<n$, we have $\{m\}\in f(a)\setminus f(b)$.  Hence $f(a)\ne f(b)$ and so $f$ is injective. Thus $|\seq (A)|\leq|\Part_{\fin} (A\cup\omega)|=|\Part_{\fin}(A)|$.

To show that $|\seq (A)|<|\Part_{\fin}(A\cup\omega)|$, suppose to the contrary that there is an injection $g\colon \Part_{\fin}(A\cup\omega)\rightarrow\seq(A)$. We shall construct a one-to-one $\alpha$-sequence of elements in $\seq(A)$ for every ordinal $\alpha$. This contradicts Hartogs' theorem.

   First, we construct a one-to-one $\omega$-sequence of elements in $\seq(A)$ recursively by
\begin{align*}
    s_0&=g([A]^1\cup\{\{2i,2i+1\}:i<\omega\}),\\ s_{n+1}&=g(f(s_n)).
    \end{align*}
    Since $[A]^1\cup\{\{2i,2i+1\}:i<\omega\}\not\in\ran(f)$ where $f$ and $g$ are injective, $\langle s_0,s_1,\dots\rangle_\omega$ is  one-to-one.

    Suppose we have already constructed a one-to-one $\alpha$-sequence $\langle s_0,s_1,\dots\rangle_\alpha$ of elements in $\seq(A)$ for an ordinal $\alpha\geq\omega$. We shall define $s_\alpha\in \seq(A)$ which is distinct from $s_i$ for all $i<\alpha$.

     Let $S=\bigcup_{\iota<\alpha}\underline{s_{\iota}}$.  If $S$ is finite, then let $a$ be the first member of a fixed denumerable subset of $A$ which is not in $S$ and let $s_{\alpha}=\langle a\rangle$.   Suppose $S$ is infinite. Define a well-ordering $<_S$ on $S$ by \[x<_{S}y\text{ if and only if }i_x<i_y\text{ or }(i_x=i_y\text{ and }j_x<j_y),\] where  $i_z$ is the first $i$ such that $z\in\underline{s_{i}}$ and $j_z$ is the first $j$ such that $s_{i_z}(j)=z$ for $z\in S$.

    Define an injection $F\colon \Part_{\fin}(S)\rightarrow\seq(A)$ by
\[F(\Pi)=g(\Pi\cup[(A\cup\omega)\setminus S]^1).\]
    Let $Y=\Part_{\fin}(S)\setminus F^{-1}[\{s_{\iota}:\iota<\alpha\}]$.  Since $|F^{-1}[\{s_{\iota}:\iota<\alpha\}]|\leq|\{s_{\iota}:\iota<\alpha\}|=|\alpha|\leq|\seq(S)|=|S|$ and $|\Part_{\fin}(S)|=2^{|S|}$, by Lemma \ref{COR}, $|Y|=|\Part_{\fin}(S)|$. Let $p\colon \Part_{\fin}(S)\rightarrow Y$ be the canonical bijection. Then $F\circ p\colon \Part_{\fin}(S)\to\seq(A)\setminus\{s_{\iota}:\iota<\alpha\}$ is injective. Let $s_{\alpha}=F\circ p([S]^1)$. Then $s_\alpha\not\in\{s_{\iota}:\iota<\alpha\}$.

    We can see that the sequence constructed by the above process is an extension of the sequences previously constructed. Thus we can construct a one-to-one sequence whose length is a limit ordinal as the union of the sequences constructed earlier.
\end{proof}
\end{thm}

Next, we shall show that if we restrict the set of all finite sequences of a set to the set of all finite sequences with bounded lengths, then the relationship holds for arbitrary infinite sets.

\begin{thm}
\label{nleq}
    For any infinite set $A$ and any natural number $n$,

    $|\seq_{\leq n}(A)|<|\Part_{\fin}(A)|$.
\begin{proof} Let $A$ be an infinite set and $n$ be a natural number.
    First, we construct an injection from $\seq_{\leq n}(A)$ to $\Part_{\fin}(A)$.

    Pick $(n+2)^2$ elements in $A$, say $a^{j}_{i}$ where $i,j\leq n+1$ and let $A_j=\{a^j_i:i\leq n+1\}$ for all $j\leq n+1$. For each $c=\langle c_0,c_1,\dots,c_{k-1}\rangle\in\seq_{\leq n}(A)$, let $j_c$ be the first $j$ such that $\underline{c}\cap A_{j}=\emptyset$  and let \[\pi_c=\{[c_{\iota}]:\iota<k\}\cup\{\{a^{j_c}_{i}:k\leq i\leq n+1\}\}\cup[A\setminus(\underline{c}\cup A_{j_c})]^1,\] where $[c_{\iota}]=\{c_{\iota}\}\cup\{a^{j_c}_{m}:m<k\text{ and }c_{m}=c_{\iota}\}$ for $\iota<k$.

    Define $f\colon  \seq_{\leq n}(A)\to\Part_{\fin}(A)$ by $f(c)=\pi_c$. To show that $f$ is  injective, let $b=\langle b_0,b_1,\dots,b_{k-1}\rangle,c=\langle c_0,c_1,\dots,c_{\ell-1}\rangle\in\seq_{\leq n}(A)$ such that $b\ne c$.

    \textbf{Case 1.} $j_b=j_c$.

    If there is $i<\min\{k,\ell\}$ such that $b_i\ne c_i$, then $a^{j_b}_i\in[b_i]\cap[c_i]$ but $[b_i]\ne[c_i]$, so $\pi_b\ne\pi_c$. Otherwise, $b$ and $c$ must be of different lengths. We may assume that $b$ is a proper sub-sequence of $c$, so $k<\ell$. Then $c_k\not\in\{a^{j_b}_i:k\leq i\leq n+1\}=[a^{j_b}_k]_{\pi_b}$ but $c_k\in[a^{j_b}_k]_{\pi_c}$, so $\pi_b\ne\pi_c$.

\smallskip

    \textbf{Case 2.} $j_b\ne j_c$.

     Then $[a^{j_b}_{n+1}]_{\pi_c}$ is a singleton or contains some $a\in A_{j_c}$ (when $a^{j_b}_{n+1}\in \underline{c}$). Since $[a^{j_b}_{n+1}]_{\pi_b}$ is a non-singleton subset of $A_{j_b}$ and $A_{j_b}\cap A_{j_c}=\emptyset$, $[a^{j_b}_{n+1}]_{\pi_c}\ne[a^{j_b}_{n+1}]_{\pi_b}$. Thus $\pi_b\ne\pi_c$.

    \medskip

    Hence $f$ is injective, so $|\seq_{\leq n}(A)|\leq|\Part_{\fin}(A)|$.

    Next, suppose to the contrary that $|\seq_{\leq n}(A)|=|\Part_{\fin}(A)|$. Then there is an injection $g\colon \Part_{\fin}(A)\rightarrow\seq_{\leq n}(A)$.

    We shall construct a one-to-one $\omega$-sequence of elements in $\seq_{\leq n}(A)$.

    Let $s_0=g(\{A_j:j\leq n+1\}\cup[A\setminus\bigcup_{j\leq n+1} A_j]^1)$ and $s_{k+1}=g(f (s_k))$ for all $k<\omega$. Since $\{A_j:j\leq n+1\}\cup[A\setminus\bigcup_{j\leq n+1} A_j]^1\not\in\ran(f )$ where $g$ and $f $ are injective, $\langle s_0,s_1,\dots\rangle_{\omega}$ is a one-to-one $\omega$-sequence of elements in $\seq_{\leq n}(A)$. Define $a_0=s_0(0)$ (or $a_0=s_1(0)$ if $s_0$ is the empty sequence) and $a_{k+1}=s_i(j)$ where $i$ is the first $l$ such that $\underline{s_l}\setminus \{a_m: m\leq k\}\neq\emptyset$ and $j$ is the first $p$ such that $s_i(p)\not\in\{a_m: m\leq k\}$. Thus $\langle a_0,a_1,\dots\rangle_\omega$ is a one-to-one $\omega$-sequence of elements in $A$, so $A$ is Dedekind infinite. This contradicts Theorem \ref{<}.
Hence, we can conclude that $|\seq_{\leq n}(A)|<|\Part_{\fin}(A)|$.
\end{proof}

\end{thm}

\section{CONSISTENCY RESULT AND SUMMARY}\label{sec3}

From a result in ZF, we know that \lq\lq $|\seq (X)|<|\Part_{\fin}(X)|$ for any Dedekind infinite set $X$\rq\rq.
Now, we show that the assumption that \lq\lq$X$ is Dedekind infinite\rq\rq\, cannot be removed.

We use
  permutation models which are models of ZFA, the Zermelo-Fraenkel set theory with atoms. Our relative consistency result can be transferred to ZF by the
Jech-Sochor First Embedding Theorem (cf. \cite[Theorem 6.1]{Jech}). First, we give a brief explanation of permutation models.

Let $A$ be a set of atoms and $\mathcal{G}$ be a group of permutations on $A$. Each $\pi\in\mathcal{G}$ is extended so that $\pi x=x$ for all pure sets $x$, i.e. sets whose transitive closure contain no atoms.
For a normal ideal $I$ of subsets of  $A$, a set $E\in I$ is a \emph{support} of $x$ if and only if $\fix(E) \subseteq \sym(x)$, where $\fix(E) = \{ \pi\in\mathcal{G} : \pi e = e \text{\ for all\ }e \in E \}$ and $\sym(x) = \{ \pi\in\mathcal{G} : \pi x = x \}$. Then we have a permutation model $\mathcal{V}$ induced by a normal ideal, defined hereditarily by
\[\mathcal{V}=\{x : x \text{\ has a support and\ } x\subseteq\mathcal{V} \}.\]
 The model we use here is the basic Fraenkel model  $\mathcal{V}_{F_0}$ which is a permutation model induced by the normal ideal $\fin (A)$ where $A$ is a denumerable set of atoms and $\mathcal{G}$ is the group of all permutations on $A$ (see \cite[Chapter 4]{Jech} for more details). First, we list some notation used in our work.

\begin{notat}
    \begin{enumerate}
    \item For any set $X$ and $E\in\fin(A)$, let $X_E=\{x\in X:\text{ $E$ is a support of $x$}\}$.
    \item Let $(a;b)$ denote the cyclic permutation of $A$ such that $a\mapsto b\mapsto a$ and fixes all elements in $A\setminus\{a,b\}$.
    \end{enumerate}
\end{notat}

\begin{thm}
    $\mathcal{V}_{F_0}\vDash |\seq^{1-1}(A)|\not\leq|\Part_{\fin}(A)|$.
\begin{proof}
     Suppose to the contrary that there is an injection $f\colon \seq^{1-1}(A)\to\Part_{\fin}(A)$ with a support $E$.
     We may assume that $E\neq\emptyset$.

     For any $\pi\in\fix(E)$ and any $x\in \seq^{1-1}(A)$, since $\pi(f(x))=(\pi f)(\pi x) = f(\pi x)$ where $f$ is injective,
\[\pi(f(x))=f(x) \leftrightarrow f(\pi x)=f(x) \leftrightarrow \pi x=x.\]
  Thus $f[\seq^{1-1}(A)_E]\subseteq \Part_{\fin}(A)_E$, so $|\seq^{1-1}(A)_E|\leq|\Part_{\fin}(A)_E|$.
     Since $\seq^{1-1}(E)\subseteq\seq^{1-1}(A)_E$, $|\seq^{1-1}(E)|\leq|\Part_{\fin}(A)_E|$.
     By Corollary \ref{A>B}, $|\seq^{1-1}(E)|>|\Part(E)|$. Thus, in order to get a contradiction, it remains to show that
     $|\Part(E)|\geq |\Part_{\fin}(A)_E|$.

     To show that $\Part_{\fin}(A)_E\subseteq\{Y\cup[A\setminus E]^1:Y\in\Part(E)\}$, let $\Pi\in\Part_{\fin}(A)_E$. Suppose $[a]_{\Pi}$ is not a singleton for some $a\in A\setminus E$. Then pick $b\in A\setminus (E\cup[a]_{\Pi})$ and let $\pi=(a;b)$. Thus $\pi$ fixes $E$ but $\pi(\Pi)\ne\Pi$ where as $E$ is a support of $\Pi$, a contradiction. Hence $[x]_{\Pi}=\{x\}$ for all $x\in A\setminus E$, i.e. $[A\setminus E]^1\subseteq\Pi$ as desired. Thus $|\Part_{\fin}(A)_E|\leq|\{Y\cup[A\setminus E]^1:Y\in\Part(E)\}|=|\Part(E)|$.
\end{proof}
\end{thm}

In conclusion, we have shown in ZF that
\begin{enumerate}
\item for any Dedekind infinite set $X$, $|\seq (X)|<|\Part_{\fin}(X)|$ (Theorem \ref{<}).
\item for any infinite set $X$ and any natural number $n$, $|\seq_{\leq n}(X)|<|\Part_{\fin}(X)|$ (Theorem \ref{nleq}).
\end{enumerate}

 It follows from the above consistency result that the condition that \lq\lq$X$ is Dedekind infinite set\rq\rq\, in Theorem \ref{<}  cannot be removed. Actually, Theorem 3.1 tells us more than that. It implies that  even the statement \lq\lq $|\seq^{1-1} (X)|\leq|\Part_{\fin}(X)|$ for any infinite set $X$\rq\rq\, is not provable in ZF. However, we are still wondering whether the statement \lq\lq $|\seq^{1-1} (X)|\neq|\Part_{\fin}(X)|$ for any infinite set $X$\rq\rq\, is provable in ZF or not?  This is left open for future research.

\bibliographystyle{ieeetr}

\end{document}